\documentclass[11pt]{amsart}

\usepackage{color}
\usepackage[dvipsnames]{xcolor}

\usepackage{ifpdf}
\ifpdf 
\usepackage[pdftex]{graphicx}   
\usepackage[pdftex,     
plainpages=false,   
breaklinks=true,    
colorlinks=true,
linkcolor=blue,
citecolor=blue,
pdftitle=BBM_representation_in_BD,
pdfauthor=ArroyoRabasa_Bonicatto
]{hyperref} 
\else 
\usepackage{graphicx}       
\usepackage{hyperref}       
\fi

\usepackage{amsfonts,amsmath}	
\usepackage{amssymb}
\usepackage{verbatim}
\usepackage{amsopn}
\usepackage[english]{babel}
\usepackage{amsthm}
\usepackage{enumerate}
\usepackage{mathrsfs}	
\usepackage{mathtools}
\usepackage{esint}
\usepackage{todonotes}

\usepackage{enumitem}
\usepackage{stmaryrd}
\usepackage{bm}
\usepackage{bbm}
\usepackage{caption}
\usepackage{nicefrac}
\makeatletter
\def\@cite#1#2{[\textbf{#1\if@tempswa , #2\fi}]}
\makeatother

\theoremstyle{definition} \newtheorem{definition}{Definition}[section]
\theoremstyle{definition} \newtheorem{remark}[definition]{Remark}
\theoremstyle{plain} \newtheorem{lemma}[definition]{Lemma}
\theoremstyle{plain} \newtheorem{proposition}[definition]{Proposition}
\theoremstyle{plain} \newtheorem{theorem}[definition]{Theorem}
\theoremstyle{plain} \newtheorem{corollary}[definition]{Corollary}
\theoremstyle{definition} 
\theoremstyle{plain} 
\theoremstyle{definition}

\DeclareMathOperator{\dist}{dist}
\DeclareMathOperator{\supp}{supp}

\DeclareMathOperator{\diam}{diam}

\newcommand{\R}{\mathbb{R}}

\newcommand{\Qbf}{\mathbf{Q}}

\newcommand{\e}{\varepsilon}

\renewcommand{\L}{\mathscr L}

\renewcommand{\L}{\mathscr L}
\renewcommand{\H}{\mathcal H}

\numberwithin{equation}{section} 

\newcommand{\Jrm}{\mathrm{J}}

\newcommand{\Mcal}{\mathcal{M}}

\newcommand{\Rcal}{\mathcal{R}}

\newcommand{\Sbb}{\mathbb{S}}

\DeclareMathOperator{\trace}{trace}

\newcommand{\set}[2]{\left\{\, #1 \ \textup{{:}}\ #2 \,\right\}}

\newcommand{\dpr}[1]{\langle #1 \rangle}

\newcommand{\sym}{\mathrm{sym}}
\newcommand{\skw}{\mathrm{skew}}

\newcommand{\toweak}{\rightharpoonup}
\newcommand{\toweakstar}{\overset{*}\rightharpoonup}

\newcommand{\toup}{\uparrow}

\newcommand{\embed}{\hookrightarrow}

\newcommand{\sbullet}{\begin{picture}(1,1)(-0.5,-2)\circle*{2}\end{picture}}
\newcommand{\frarg}{\,\sbullet\,}

\newcommand{\eps}{\varepsilon}

\theoremstyle{plain} \newtheorem*{theorem*}{Theorem}
\theoremstyle{plain} 
\theoremstyle{plain} \newtheorem*{mthm*}{Main Theorem}
\theoremstyle{plain} \newtheorem*{conjecture*}{Conjecture}
\theoremstyle{plain} 
\theoremstyle{plain} \newtheorem*{problem*}{Problem}

\numberwithin{equation}{section} 

\makeatletter
\@namedef{subjclassname@2020}{\textup{2020} Mathematics Subject Classification}
\makeatother

\title[A BBM representation for $BD$]{A Bourgain--Brezis--Mironescu representation \\ for functions with bounded deformation}

\author{Adolfo Arroyo-Rabasa}
\address[A.\ Arroyo-Rabasa]{Mathematics Institute, University of Warwick,
	Zeeman Building, CV4~7HP Coventry, UK}
\email{adolforabasa@gmail.com}
\author{Paolo Bonicatto}
\address[P.\ Bonicatto]{Mathematics Institute, University of Warwick,
	Zeeman Building, CV4~7HP Coventry, UK}
\email{Paolo.Bonicatto@warwick.ac.uk}

\begin{document}
	\date{\today} 
	
		\begin{abstract} We establish a non-local integral difference quotient  representation for symmetric gradient semi-norms in $BD(\Omega)$ and $LD(\Omega)$, which does not require the manipulation of distributional derivatives. Our representation extends the formulas for the symmetric gradient established by Mengesha for vector-fields in $W^{1,p}(\Omega;\R^d)$, which are inspired by the gradient semi-norm formulas introduced by Bourgain, Brezis and Mironescu in $W^{1,p}(\Omega)$ and by D\'avila in $BV(\Omega)$.
		\end{abstract}
	
	\maketitle	

\section{Introduction} Let $\Omega$ be a connected open subset of $\R^d$ with uniformly Lipschitz boundary and let $M^{d\times d}_\sym$ be the space of symmetric $(d \times d)$  real valued matrices.
The distributional symmetric gradient of an integrable vector-field $u : \Omega \to \R^d$ is defined as the $M^{d \times d}_\sym$-valued distribution
\begin{equation}\label{eq:E}
	\begin{split}
	Eu & \coloneqq \frac 12 (Du + Du^T) \\
	&  \phantom{:}= \frac 12 (\partial_i u^j + \partial_j u^i), \qquad i,j = 1,\dots,d,
	\end{split}
\end{equation}
where $\partial_i$ denotes the distributional partial derivative in the $e_i$ canonical direction of $\R^d$.  Analogously to the classical Sobolev spaces, one may define spaces of functions with $L^p$ symmetric gradients as follows: if $p \in [1,\infty)$, then
\[
	LD^p(\Omega) \coloneqq \set{u \in L^p(\Omega;\R^d)}{Eu \in L^p(\Omega;M^{d \times d}_\sym)}
\] is the space of $p$-integrable vector-fields $u$ such that $Eu$ can be represented by a $p$-integrable $M^{d\times d}_\sym$-valued field on $\Omega$.\footnote{Or simply $LD(\Omega)$ when $p = 1$.} The introduction of more general spaces, where one considers functions with symmetric gradients represented by a symmetric matrix-valued Radon measure, has been a crucial landmark in the understanding of plasticity and fracture models in linear elasticity (we refer the interested reader to~\cite{anz,matt,kohntemam,suquet,temam1} and references therein for a deeper discussion on this topic). Precisely for such purposes, Christiansen, Matthies and Strang~\cite{matt} and Suquet~\cite{suquet} independently introduced the space
\[
	BD(\Omega) \coloneqq \set{u \in L^1(\Omega;\R^d)}{Eu \in \Mcal_b(\Omega;M^{d \times d}_\sym)},
\]
of \emph{functions with bounded deformation} over $\Omega$, which consists of all integrable vector-fields $u : \Omega \to \R^d$ such that {their distributional symmetric gradient}  $Eu$ can be represented by a $M^{d\times d}_\sym$-valued \emph{bounded Radon measure}. 

Functions in $BD(\Omega)$ possess \emph{similar} functional and fine properties to the {ones} exhibited by functions in $BV(\Omega;\R^d)$ (see, e.g., \cite{ACDM,baba,haj,Kohn}). Here, 
	\[
	BV(\Omega;\R^d) =  \set{u \in L^1(\Omega;\R^d)}{Du \in \Mcal_b(\Omega;M^{d \times d})},
	\]
	is the space of \emph{vector-fields with bounded variation} over $\Omega$. However, the kernel of $E$ is strictly larger than the kernel of the gradient operator $D$. This property substantially separates these two operators from a functional viewpoint. Indeed, a vector-field $u : \Omega \to \R^d$ satisfies $Eu = 0$ in the sense of distributions on $\Omega$ if and only if $u$ is a rigid motion, i.e., $u = Rx + c$, where $R \in M^{d\times d}_\skw$ is a $(d \times d)$ skew symmetric matrix and $c \in \R^d$. For this reason, one cannot expect, in general, to control $Du$ in terms of $Eu$ alone. In order to control the $L^p$ norm of $Du$ in terms of the one of $Eu$, one has to translate by all possible rigid motions (modulo constant displacements). This reasoning applies \emph{only} when we restrict ourselves to the range $p \in (1,\infty)$, as it is reflected in the following version of Korn's inequality
\[
	\inf_{R \in M^{n\times n}_\skw} \|Du - R\|_{L^p(\Omega)} \le K(\Omega,p)  \|Eu\|_{L^p(\Omega)}, \qquad p \in (1,\infty).\footnote{This (rigidity) version is a consequence of Korn's second inequality 
		\[
		\|u\|_{W^{1,p}(\Omega)} \le C \left(\|u\|_{L^p(\Omega)} + \|Eu\|_{L^p(\Omega)}\right),
		\]
		whose first proof is arguably contained as a particular case of the coercive estimates established by Smith~\cite{smith}.}
\] 
As a consequence, for $p > 1$, the definition of $LD^p(\Omega)$ is superfluous from a functional point of view, as it is straightforward to verify that $LD^p(\Omega)$ coincides with the Sobolev space $W^{1,p}(\Omega;\R^d)$. On the other hand, as $p \to~1^+$, the optimal constant $K(\Omega,p)$ in Korn's inequality blows-up to infinity and this points at the fact that the symmetric gradient and the gradient \emph{are truly different} operators from a functional perspective. This is formalized through Ornstein's non-inequality~\cite{ornstein}, which conveys that neither $LD(\Omega)$ embeds into $W^{1,1}(\Omega;\R^d)$, nor $BD(\Omega)$ embeds into the space $BV(\Omega;\R^d)$. 

The goal of this paper is to prove a limiting non-local integral formula for a total variation semi-norm of the symmetric gradient of an integrable vector-field, which avoids the direct manipulation of the distributions in~\eqref{eq:E}. The results presented here are inspired by similar formulas for gradients first established by Bourgain, Brezis and Mironescu for functions in $W^{1,p}(\Omega)$, by D\'avila for $BV(\Omega)$, and for the symmetric gradient operator by Mengesha for functions $W^{1,p}(\Omega;\R^d)$. Therefore, to contextualize and motivate our findings, we shall first recall the theory for gradients.

\subsection{Background theory for the gradient operator}In~\cite{BBM}, Bourgain, Brezis and Mironescu established the following (BBM) limiting difference quotient representation for Sobolev functions: if  $u \in W^{1,p}(\Omega)$ for some $p \in [1,\infty)$, then 
\begin{equation}\label{eq:BBM}
\lim_{\eps \to 0^+} \iint_{\Omega \times \Omega} \frac{|u(y) - u(x)|^p}{|y - x|^p} \, \rho_\eps(y-x) \,d y \,d x= K_{p,d} \int_\Omega |\nabla u|^p,
\end{equation}
	where {$K_{p,d}$ is a positive constant depending on $p,d$ and}  $\{\rho_\eps\}_{\eps > 0} \subset L^1(\R^d)$ is a family of  non-negative radial probability mollifiers 
	\begin{equation}\label{eq:eps1}
		\|\rho_\eps\|_{L^1(\R^d)}  = 1, \qquad \rho_\eps(x) = \hat \rho_\eps(|x|),
	\end{equation}
	which approximate the Dirac mass at zero in the sense that
	\begin{equation}\label{eq:eps}
		\lim_{\eps \to 0^+} \|\rho_\eps\|_{L^1(\R^d \setminus B_\delta)} 
		= 0 \qquad \text{for all} \; \delta > 0.
	\end{equation}
The authors also show that the converse holds in the range $p \in (1,\infty)$. More precisely,  that if $u \in L^p(\Omega)$ and 
\begin{equation}\label{eq:limsup}
	\liminf_{\eps \to 0^+} \iint_{\Omega \times \Omega} \frac{|u(y) - u(x)|^p}{|y - x|^p} \, \rho_\eps(y-x) \,d y \,d x < \infty, 
\end{equation}
then automatically $u \in W^{1,p}(\Omega)$.
The analysis of the limiting case $p = 1$ is more delicate since one must take into account the appearance of mass concentrations in the gradient. In this regard, D\'avila~\cite{D} established a related representation for $BV(\Omega)$. He proved that $u \in BV(\Omega)$ if and only if $u \in L^1(\Omega)$ and~\eqref{eq:limsup} holds with $p = 1$; in that case the limit exists and is given by
\begin{equation}\label{eq:D}
\lim_{\eps \to 0^+} \iint_{\Omega \times \Omega} \frac{|u(y) - u(x)|}{|y - x|} \, \rho_\eps(y-x) \,d y \, d x = K_{1,d} |Du|(\Omega),
\end{equation}
where $|Du|$ is the total variation measure associated to $Du \in \Mcal_b(\Omega;\R^d)$. 

\subsection{Background theory for the symmetric gradient}Now that we have recalled the representations for Sobolev and bounded variation functions, we shall center on the theory concerning the symmetric gradient operator. In order to do this, we shall first introduce an auxiliary family of $p$-norms in the $M^{d \times d}_\sym$ as follows: for $p \in [1,\infty)$, we set
\[
\Qbf_p(A) \coloneqq \left(\fint_{\Sbb^{d-1}} |\dpr{A\omega,\omega}|^p \, dS(\omega)\right)^\frac{1}{p}, \qquad A \in M^{d \times d}_\sym, 
\]  
where $S$ stands for the surface measure on the unit sphere $\Sbb^{d-1}$ in $\R^d$.
It is easy to check that $\Qbf_p(\cdot)$ defines a norm on $M^{d \times d}_\sym$ (more details will be given in Sect.~\ref{sec:Q}). 
Furthermore, for $p \in [1,\infty)$, a vector-field $u \in L^p(\Omega; \R^d)$ and a Borel set $U \subset \R^d$, we introduce the following short-hand notation:
\begin{equation*}
	\mathscr F_{p,\e}(u,U) :=  \iint_{U \times U}   \frac{\vert \langle u(y) - u(x),y-x \rangle |^p}{|y - x|^{2p}} \, \rho_\varepsilon(y-x) \, dy\, dx, 
\end{equation*}
where the right-hand side is well-defined as the integral of a non-negative function and may take the value $\infty$. 

In~\cite[Theorem~2.2]{mengesha} Mengesha proved (under slightly more restrictive assumptions) the following analogue of~\eqref{eq:BBM} for the symmetric gradient and functions in $W^{1,p}(\Omega;\R^d)$:

\begin{theorem}\label{thm:m}Let $p \in (1,\infty)$ and let $u \in L^p(\Omega; \R^d)$. Then, $u$ belongs to $W^{1,p}(\Omega;\R^d)$ if and only if 
	\begin{equation}\label{eq:functional_with_p}
		 \liminf_{\varepsilon \to 0^+} \mathscr F_{p,\e}(u,\Omega) < \infty. 
	\end{equation}
	Moreover, the (extended) limit always exists and equals 
		\begin{equation*}
		\lim_{\varepsilon \to 0^+} \iint_{\Omega \times \Omega}   \frac{\vert \langle u(y) - u(x),y-x \rangle |^p}{|y - x|^{2p}} \, \rho_\varepsilon(y-x) \, dy\, dx = \int_\Omega \Qbf_p(Eu(x))^p \,d x,
		\end{equation*}
with the convention that the right-hand side integral equals $\infty$ whenever $u \notin W^{1,p}(\Omega;\R^d)$.\footnote{Notice that our $\Qbf_p$-norms differ by a multiplicative constant $|\Sbb^{d-1}|$ with respect to Mengesha's original norms. This, however, seems to stem from a minor normalization typo.}
\end{theorem}

Following Dávila's ideas, and as a direct consequence of Mengesha's representation, one has the following strong compactness and convergence result (which is somehow implicit in~\cite{mengesha}):

\begin{corollary}\label{cor:p}
	Let $p \in (1,\infty)$ and let $u \in L^p(\Omega;\R^d)$. The following are equivalent:
	\begin{enumerate}
		\item[(a)]$u \in W^{1,p}(\Omega;\R^d)$,
		\item[(b)] the family of functions
		\[
		\mu_{p,\eps}(x) \coloneqq \left(\int_\Omega \frac{\vert \langle u(y) - u(x),y-x \rangle |^p}{|y - x|^{2p}} \, \rho_\varepsilon(y-x) \, dy\right)^\frac{1}{p}, \quad \eps \in (0,1),
		\]
		is uniformly bounded in $L^p(\Omega)$,
		\item[(c)] $\mu_{p,\eps} \longrightarrow \Qbf_p(Eu)$ in $L^p(\Omega)$ as $\eps \to 0^+$.
	\end{enumerate}
\end{corollary}

\begin{remark}[An alternative proof]
	For the convenience of the reader and since our proofs depart in crucial points from the ones given by Bourgain, Brezis and Mironescu, Davila and Mengesha, we have decided to include here the proof of Theorem~\ref{thm:m} and of Corollary \ref{cor:p}. Notice also that we do not require $\Omega$ to be bounded in any of these or the forthcoming results.
\end{remark}

In addition to Theorem~\ref{thm:m}, Mengesha proved the following criterion for functions of bounded deformation in terms of the symmetric difference quotient energy: a map $u$ belongs to $BD(\Omega)$ if and only if $u \in L^1(\Omega;\R^d)$ and~\eqref{eq:functional_with_p} is finite for $p = 1$. More precisely, he showed that there exist positive constants $\gamma_1,\gamma_2$ satisfying
\begin{equation}\label{eq:m2}
	\begin{split}
		\gamma_1 \|u\|_{BD(\Omega)}  & \le \liminf_{\eps \to 0^+} \mathscr F_{1,\e}(u,\Omega) \\
		& \le \limsup_{\eps \to 0^+} \mathscr F_{1,\e}(u,\Omega) \le \gamma_2\|u\|_{BD(\Omega)}, 
	\end{split}
\end{equation}
with the convention that these norms may take the value $\infty$ and where
 $\|u\|_{BD(\Omega)} \coloneqq \| u\|_{L^1(\Omega)} + |Eu|(\Omega)$ is the standard norm in $BD(\Omega)$ (see below). As it is already suggested by~\eqref{eq:D} and~\eqref{eq:m2}, the analysis and characterization of~\eqref{eq:functional_with_p} (with $p = 1$), requires one to relax the the statement to functions with bounded deformation, rather than to elements of $LD(\Omega)$ or $W^{1,1}(\Omega)$. In other words, the sufficiency of the first statement of Theorem~\ref{thm:m} fails for $p=1$ because one must take into account the appearance of mass concentrations on the symmetric gradient. 

\subsection{Main results}In order to state our results, we need to recall the following basic geometric measure theory facts: if $u \in BD(\Omega)$, then $Eu$ is a bounded $M^{d\times d}_\sym$-valued Radon measure and hence, by Riesz' representation theorem and the Radon--Nikod\'ym differentiation theorem, we may write $Eu$ in polar form
\[
Eu = e_u \,|Eu| \qquad \text{as measures on $\Omega$,}
\]
where $|Eu| \in \Mcal^+(\Omega)$ is the total variation measure of $Eu$ (induced by the classical Frobenius inner product of matrices)
and
\[
e_u(x) \coloneqq \frac{d\, Eu}{d\, |Eu|}(x) = \lim_{r \to0^+} \frac{Eu(B_r(x))}{|Eu|(B_r(x))}, \qquad x\in \Omega,
\]
is a norm-1 density function in $L^\infty(\Omega,|Eu|;M^{d \times d}_\sym)$. We may then define the $\Qbf_1$-total variation measure of $Eu$ as 
\[
[Eu](U) \coloneqq \int_U \Qbf_1(e_u(x)) \, d|Eu|(x), \qquad \text{$U\subset \Omega$ Borel.}
\]

Having set this notation, we are finally ready to state our main result:
	\begin{theorem}\label{thm:1} 
		Let $u \in L^1(\Omega; \R^d)$. Then, the (extended) limit
		\[
			\lim_{\eps \to 0^+} \mathscr F_{1,\eps}(u,\Omega)  \in [0,\infty]
		\]
		always exists and equals
		\begin{equation*}
			\lim_{\varepsilon \to 0^+} \iint_{\Omega \times \Omega}   \frac{\vert \langle u(y) - u(x),y-x \rangle |}{|y-x|^2} \, \rho_\varepsilon(y-x) \, dy\,dx = [Eu](\Omega),
		\end{equation*}
		with the convention that the right-hand side equals $\infty$ whenever $u \notin BD(\Omega)$.
			\end{theorem}
		In particular, we obtain the following strengthening of~\eqref{eq:m2}:
		\begin{corollary}
			There exists a dimensional constant $C_{d}>0$ such that  
			\begin{equation*}
				C_{d}|Eu|(\Omega) \le \lim_{\eps \to 0^+} \mathscr F_{1,\e}(u,\Omega) \le |Eu|(\Omega),
		\end{equation*}
	under the convention that the semi-norms may attain the value $\infty$.
			\end{corollary}

\begin{remark}{One can draw} a parallelism between the $BV$-theory and the $BD$-theory in the following way: Theorem~\ref{thm:1} extends Theorem~\ref{thm:m} to $BD(\Omega)$, just as Dávila's representation~\eqref{eq:D} extends~\eqref{eq:BBM} to $BV(\Omega)$.
\end{remark} 

\begin{remark}
	Notice that {if} $u \in LD(\Omega)$, we still get 
	\begin{equation*}
		\lim_{\varepsilon \to 0^+} \iint_{\Omega \times \Omega}   \frac{\vert \langle u(y) - u(x),y-x \rangle |}{|y - x|^{2}} \, \rho_\varepsilon(y-x) \, dy\, dx = \int_\Omega \Qbf_1(Eu(x)) \,d x.
	\end{equation*}

\end{remark}

In general, Corollary~\ref{cor:p} does not have an $L^1$-convergence analog. However, we can still deduce the following compactness and strict convergence (in the sense of measures) results:
\begin{corollary}\label{cor:BD}
	Let $u \in L^1(\Omega;\R^d)$. The following are equivalent:
	\begin{enumerate}
		\item[(a)]$u \in BD(\Omega)$,
		\item[(b)] the family of functions
		\[
		\mu_{1,\eps}(x) \coloneqq \int_\Omega \frac{\vert \langle u(y) - u(x),y-x \rangle |}{|y - x|^{2}} \, \rho_\varepsilon(y-x) \, dy, \quad \eps \in (0,1),
		\]
		is uniformly bounded in $L^1(\Omega)$,
		\item[(c)] $\mu_\eps \,\mathscr L^d \toweakstar  [Eu]$ as measures in $\Mcal(\Omega)$ and
		\[
			\mu_{1,\eps}(\Omega) \longrightarrow [Eu](\Omega) \qquad \text{as $\eps \to 0^+$}.
		\]
	\end{enumerate}
\end{corollary}

We close the exposition of our results with some  consequences of the representation for $BD$-spaces that are inspired by the work of Ponce and Spector~\cite{ponce} on $BV$-spaces. 
Let us recall (see~\cite{ACDM,haj}) that a function $u \in BD(\Omega)$ is approximately differentiable almost everywhere, that is, there exists a measurable matrix-field $x \mapsto \mathrm{ap}\,\nabla u(x) \in M^{d \times d}$ satisfying
\[
	\lim_{r \to 0^+} \fint_{B_r(x)} \frac{|u(y) - u(x) - \mathrm{ap}\,\nabla u(x)[y-x]|}{r} = 0 \quad \text{for $\mathscr L^d$-a.e. $x \in \Omega$.}
\]
In this case, the matrix $\mathrm{ap} \,\nabla u(x)$ is called the \emph{approximate differential} of $u$ at $x$.
On the other hand, it is also well-known (see~\cite{ACDM}) that if $u \in BD(\Omega)$, then $Eu$ can be decomposed into an absolutely continuous and a singular part as 
\[
	Eu = \mathcal Eu\, \mathscr L^d + E^s u, \quad \mathcal Eu(x) \coloneqq \frac 12(\mathrm{ap}\,\nabla u(x) + \mathrm{ap}\,\nabla u(x)^T),
\]
where $|E^s u| \perp \mathscr L^d$.

Following verbatim the ideas contained in~\cite[Sect.~2]{ponce}, we give a criterion for the absolute continuity of symmetric gradient measures in the terms of its \emph{approximate first-order Taylor expansion}:
\begin{corollary}
	Let $u \in BD(\Omega;\R^d)$ and let $\mathrm{ap}\,\nabla u(x) \in M^{d \times d}$ denote the approximate differential of $u$ at  a point $x$, which exists $\mathscr L^d$-almost everywhere in $\Omega$. Then, the extended limit
	\begin{equation*}
		 	\lim_{\varepsilon \to 0} \iint_{\Omega \times \Omega}   	\frac{\vert \dpr{u(y) - u(x) - \mathrm{ap}\,\nabla u(x)[y-x],y-x}|}{|y-x|^2} \, \rho_\varepsilon(y-x) \, dy \, dx
	\end{equation*}
exists and equals $[E^s u](\Omega)$.

In particular, $u \in LD^1(\Omega)$ if and only if
	\[
		\lim_{\varepsilon \to 0} \iint_{\Omega \times \Omega}   	\frac{\vert \dpr{u(y) - u(x) - F(x)[y-x],y-x}|}{|y-x|^2} \, \rho_\varepsilon(y-x) \, dy \, dx = 0
	\]
	for some Borel measurable function $F\colon  \Omega \to M^{d \times d}$.
\end{corollary}
\begin{remark}
	The assertions of the previous corollary remain unchanged if instead we consider the integrand
	\[
		\frac{\vert \dpr{u(y) - u(x) - \mathcal E u(x)[y-x],y-x}|}{|y-x|^2} \, \rho_\varepsilon(y-x).
	\]
\end{remark}

\section{Preliminaries}
In this section we briefly recall the properties of the $\Qbf$-norms defined in the introduction and we also review some well-known results about $BD(\Omega)$ and $W^{1,p}(\Omega;\R^d)$ spaces, where $\Omega$ is a Lipschitz (possibly unbounded) open set of $\R^d$. 

In all that follows, we write $|\frarg|$ to denote the classical Frobenius inner product norm on $M^{d \times d}$ (and $M^{d\times d}_\sym$), that is, 
\[
	|A|^2 \coloneqq \trace(A^T A) = \sum_{i,j=1}^d A_{ij}^2, \qquad A = (A_{ij}).
\] 

\subsection{Strict convergence}
We say that a sequence  $(u_k) \in BD(\Omega)$ converges \emph{strictly} to $u$ in $BD(\Omega)$ provided that
\[
	u_k \rightarrow u \; \text{in $L^1(\Omega;\R^d)$}, \quad Eu_k \toweakstar Eu \; \text{in $\Mcal(\Omega;M^{d \times d}_\sym)$},
\]
and 
\[
	|Eu_k|(\Omega) \to |Eu|(\Omega).
\]
To denote this, we write
\[
	u_k \stackrel{s}\longrightarrow u \quad \text{in $BD(\Omega)$} 
\]
\subsection{Extension operators}
When $p > 1$,  a direct consequence of Korn's inequality is the embedding $$LD^p(\Omega) \embed W^{1,p}(\Omega;\R^d).$$ This, in particular, allows one to make use of a plethora of extension operators $T\colon W^{1,p}(\Omega;\R^d) \to W^{1,p}(\R^d;\R^d)$ whenever $p > 1$. If $p = 1$, it is well known that neither $LD^1(\Omega)$ nor $BD(\Omega)$  embed into $W^{1,1}(\Omega;\R^d)$, not even locally. However, $BD(\Omega)$ does possess trace operators~\cite{baba} and in particular it possesses an extension operator $T : BD(\Omega) \to BD(\R^d)$ that does not charge the boundary, i.e., such that 
\[
	|E(Tu)|(\partial \Omega) = 0.
\]

\subsection{The $\Qbf$-norms on $LD^p$ and $BD$}\label{sec:Q}
As it has already been advanced in the previous section, we will work with certain Rayleigh-type norms on $M^{d\times d}_\sym$. For the convenience of the reader, let us recall its definition:
\begin{definition} Let $p \in [1,\infty)$. We define a norm on $M^{d \times d}_\sym$ by letting
	\begin{align*}
		\mathbf Q_p(A) & :=  \bigg(\fint_{\mathbb S^{d-1}} | \langle A\omega, \omega \rangle |^p \, dS(\omega)\bigg)^\frac{1}{p} \\[10pt]
		& \phantom{:}= \kappa_{p,d}\|\dpr{A\frarg,\frarg}\|_{L^p(\Sbb^{d-1})}, 
	\end{align*}
where $\kappa_{p,d} := |\Sbb^{d-1}|^{-1/p}$ and $|\Sbb^{d-1}|$ is the measure of the $(d-1)$-dimensional sphere in $\R^d$. 
\end{definition}
	That $\mathbf Q_p$ defines a norm for every $p \in [1,\infty)$ is an immediate consequence of the triangle inequality in $L^p$ and the spectral theorem for matrices in $M^{d \times d}_\sym$. Indeed, it is straightforward to verify that $\Qbf_p$ is invariant under the conjugation with orthogonal matrices and hence
	\begin{equation}\label{eq:eigen}
	\Qbf_p(A)^{p} = \fint_{\mathbb S^{d-1}} (\lambda_1\omega_1^2 + \dots \lambda_d\omega_d^2)^p \, dS(\omega),
	\end{equation}
	where $\lambda_1,\dots,\lambda_d$ are the eigenvalues of $A$. In particular, $\Qbf_p(A) = 0$ if and only if all the eigenvalues are zero, and by homogeneity and Cauchy--Schwarz' inequality it follows that
	\[
	C_{d,p} |A| \le \Qbf_p(A) \le |A|,
	\]
	for some constant $C_{d,p}$.
	In a natural manner, this defines an equivalent norm for functions  $F \in L^p(\Omega;M^{d \times d}_\sym)$ by setting
	\[
	[F]_p(\Omega) \coloneqq \bigg(\int_{\Omega} \Qbf_p(F)^p \; dx \bigg)^\frac{1}{p}\,.
	\]
	For $(M^{d \times d}_\sym)$-valued Radon measures $\mu \in \Mcal_b(\Omega;M^{d \times d}_\sym)$, we may consider the $\Qbf_1$-variation measure $[\mu] \in \Mcal^+(\Omega)$, which on Borel sets $U \subset \Omega$ is defined as the non-negative Radon measure taking the values
	\[
		[\mu](U) \coloneqq \int_U \Qbf_1\left(\frac{\mu}{|\mu|}(x)\right) \, d|\mu|(x),
	\]
	where $\mu/|\mu|$ is the Radon--Nikod\'ym derivative of $\mu$ with respect to $|\mu|$.
	Notice that $[\frarg]$ and $|\frarg|$ are equivalent norms  in $\Mcal_b(\Omega;M^{d \times d}_\sym)$. Hence, both 
	\begin{align*}
		|u|_{LD^p(\Omega)} &\coloneqq \|u\|_{L^p(\Omega)} + [Eu]_p(\Omega), \\
		|u|_{BD(\Omega)} &\coloneqq \|u\|_{L^1(\Omega)} + [Eu](\Omega),
	\end{align*}
	define equivalent norms on $LD^{p}(\Omega;\R^d)$ and $BD(\Omega)$ respectively.

\begin{remark}[Strict convexity and strict convergence]
Every $\Qbf_p$ is a convex $1$-homogenenous function (each of these being norms). However, it is worthwhile to mention that $\Qbf_1$ \emph{is not} a strictly convex norm. Indeed, it can be seen from~\eqref{eq:eigen} that $\Qbf_1(A) = \alpha_d \, \mathrm{tr} (A)$ for all positive definite matrices $A \in M^{d \times d}_\sym$,
which implies that this norm behaves linearly on this connected open set of matrices. In particular, $\Qbf_1$ \emph{is not} the norm associated to an inner product on $M^{d \times d}_\sym$, and the $[\frarg]$-strict convergence of measures
\[
\mu_\eps \toweakstar \mu \quad \text{in $\Mcal(\Omega;M^{d \times d}_\sym)$}, \qquad[\mu_\eps](\Omega) \to [\mu](\Omega),
\]
does not necessarily imply that $\mu_\eps \longrightarrow \mu$ in the classical $|\frarg|$-strict sense of measures; this last assertion follows from~\cite[Theorem~1.3]{VS}.
\end{remark}

\section{Proof of the main result} 

\subsection{Proof of the upper bound}
The proof of the upper bound inequality is somewhat standard as it follows closely the ideas from~\cite{BBM} and~\cite{D}, with the exception that we are considering slightly different integrands here. 

The first step will be to show a suitable $\eps$-independent upper bound for $u \in C^1(\R^d;\R^d)$ (see Lemma~\ref{lem:1} below). Once the scale-independent bound contained in Lemma~\ref{lem:1} has been established, the sought upper bound for $u \in LD^p(\Omega)$ and $u \in BD(\Omega)$ will follow from the existence of suitable extension operators for these spaces. 

For the next lemma, we write 
\begin{equation}\label{eq:measure_mu}
\mu_{p,\varepsilon}(x) := \left(\int_{\R^d}   \frac{\vert \langle u(y) - u(x),y-x \rangle |^{p}}{|y-x|^{2p}} \, \rho_\varepsilon(y-x) \, dy\right)^\frac{1}{p} . 
\end{equation}

\begin{lemma}\label{lem:1} Let $u \in C^1(\R^d;\R^d)$ and let $U \subset \R^d$ be a Borel set. For a positive radius $R > 0$, we shall write 
	\[
	U_R \coloneqq U + B_R,
	\]
	to denote the set whose complement is at distance $R$ from $U$. Then, it holds 
	\begin{equation}\label{eq:upper_bound}
	\int_U\mu_{p,\varepsilon}^p \le \left([Eu]_p(U_R)\right)^p + \frac{2}{R^p} \Vert u \Vert_{L^p(U)}^p \int_{\R^d \setminus B_R(0)} \rho_\varepsilon(x)\, dx. 
	\end{equation}
\end{lemma}

\begin{proof}[Proof of the lemma]For a fixed $R>0$, we split 
	\begin{equation*}
	\int_U\mu_{p,\eps}^p = \int_U \int_{\R^d}   \frac{\vert \langle u(y) - u(x),y-x \rangle |^{p}}{|y-x|^{2p}} \, \rho_\varepsilon(y-x) \, dy \, dx = \  I_1 + I_2,  
	\end{equation*}
	where 
	\begin{equation*}
	I_1 := \int_U  \int_{B_R(x)} \frac{|\langle u(y) - u(x),y-x \rangle |^p}{|y - x|^{2p}} \, \rho_\varepsilon(y-x) \, dy \, dx
	\end{equation*}
	\begin{equation*}
	I_2 := \int_U  \int_{\R^d \setminus B_R(x)} \frac{|\langle u(y) - u(x),y-x \rangle |^p}{|y - x|^{2p}} \, \rho_\varepsilon(y-x) \, dy \, dx. 
	\end{equation*}
	Clearly, the second term in the right hand side of \eqref{eq:upper_bound} is an upper bound for $I_2$. We shall hence focus on showing that $\left([Eu]_p(U_R)\right)^p$ is an upper bound for $I_1$. To this end, let us recall  the path integral identity
	\begin{equation*}
	u(y) -u(x) = \int_0^1 \nabla u (ty + (1-t)x) \cdot (y-x) \, dt, \qquad y,x \in \R^d.
	\end{equation*}
	Together with Jensen's inequality, this yields that
	\begin{equation*}
	\int_U\int_{B_R(x)} \int_0^1 \left \vert \left \langle \nabla u (ty + (1-t)x) \cdot  \frac{y-x}{|y-x|},   \frac{y-x}{|y-x|} \right \rangle \right \vert^p dt\, \rho_\varepsilon(y-x) \, dy \, dx
	\end{equation*}
	is an upper bound for $I_1$.
	
	Fixing $x$, we apply the change of variables $h := y-x$ and apply Tonelli's Theorem to permute the integrals and obtain
	\begin{equation*}
	\begin{split}
	I_1 & \le \int_{B_R} \int_0^1 \int_U \left \vert  \left \langle \nabla u (x + th) \cdot \frac{h}{|h|},   \frac{h}{|h|} \right \rangle \right \vert^p  dx\, dt\, \rho_\varepsilon(h) \, dh \\ 
	& = \int_{B_R} \int_0^1 \int_{U+th} \left \vert  \left \langle \nabla u (z) \cdot \frac{h}{|h|},   \frac{h}{|h|} \right \rangle \right \vert^p  dz\, dt\, \rho_\varepsilon(h) \, dh \\
	& \le \int_{B_R} \int_{U_R} \left \vert  \left \langle \nabla u (z) \cdot \frac{h}{|h|},   \frac{h}{|h|} \right \rangle \right \vert^p  dz\, \rho_\varepsilon(h) \, dh.
\end{split}
\end{equation*}
Observe that if $A \in M^{d \times d}$, then $\dpr{A\omega,\omega} = \frac 12\dpr{(A^T + A)\omega,\omega}$ for all $\omega \in \R^d$. We shall use this to express the integrand on the right-hand side of the estimate in terms of $Eu(x)$ rather than $D u(x)$. Therefore, from the change of variables $\omega = \frac{h}{|h|}$, the coarea formula on balls and the radial symmetry of the mollifier we deduce that
\begin{align*}
	I_1 & \le |\partial B_r|\int_{0}^R  \widehat{\rho_\varepsilon}(r) \, r^{d-1}dr \times \bigg( \int_{U_R} \int_{\mathbb S^{d-1}} \left \vert  \left \langle Eu (z)  \omega,   \omega \right \rangle \right \vert^p  d\H^{d-1}(\omega)\, dz \bigg)  \\
	& \le \|\rho\|_{L^1} \int_{U_R} \Qbf_p(Eu)^p \, dz  \le \left([Eu]_p(U_R)\right)^p.
\end{align*}
This completes the proof of the lemma.
\end{proof}

	\begin{proof}[Proof of the upper bound] Let $U \subset \Omega$ be an open set. Let $u \in LD^p(\Omega)$ or $u \in BD(\Omega)$. We aim to show that
		\[
			 \limsup_{\varepsilon \to 0} \mathscr F_{p,\eps}(u,\Omega) \le [Eu]_p(\Omega)^p \quad \text{when $p > 1$}
		\] 
		or 
		\[
		\limsup_{\varepsilon \to 0} \mathscr F_{1,\eps}(u,\Omega) \le [Eu](\Omega),
		\] 
		respectively.
		Let us recall from the preliminaries that for $p>1$ there exists an extension operator $T: W^{1,p}(\Omega;\R^d) \to W^{1,p}(\R^d;\R^d)$, and, for $p = 1$, there also exists an extension operator,  which for the sake of simplicity we shall also denote by $T\colon BD(\Omega) \to BD(\R^d)$, that does not charge the boundary, that is, $|E(Tu)|(\partial \Omega) = 0$. On either case, a standard mollification argument yields an approximating sequence $(u_k) \subset C^\infty(\R^d;\R^d)$ satisfying 
\[
	u_k \longrightarrow Tu \; \text{in $W^{1,p}(\R^d)$ \qquad when $p > 1$},
\]
or
\[
	u_k \stackrel{s}\longrightarrow Tu \quad \text{in $BD(\R^d)$ \qquad when $p = 1$.}
\]
In the latter case, there exists a full $L^1$-measure set $I \subset (0,\infty)$ for which it holds $|Eu|(\partial U_R) = 0$ for all $R \in I$. In particular, from the strict convergence above we get
\begin{align*}
		u_k \stackrel{s}\longrightarrow Tu \quad \text{in $BD(U_R)$ \qquad \text{for all $R \in I$.}}
\end{align*}

\emph{Claim.} Let $R \in I$ and let $\eps > 0$. Then
	\begin{equation}\label{eq:transition}
		\lim_{k \to \infty}  \mathscr F_{p,\eps}(u_k,U_R)= \mathscr F_{p,\eps}(Tu,U_R)
	\end{equation}
and
\begin{equation}\label{eq:transition2}
	\lim_{k \to \infty} [Eu_k]_p(U_R) = M(u,U_R) \coloneqq
	\begin{cases}
		[E(Tu)]_p(U_R) & \text{if $u \in LD^p(U)$},\\
		[E(Tu)](U_R)& \text{if $u \in BD(U)$}.
	\end{cases}
\end{equation}
	The first limit follows directly from Tonelli's theorem and the strong convergence $u_k \to u$ in $L^p(U)$. 
	Let us  address the convergence for the second limit.
	For $p = 1$, the argument follows directly from the strict convergence $u \stackrel{s}\longrightarrow Tu$ in $BD(U_R)$, Reshetnyak's continuity theorem (\cite[Thm.~2.39]{AFP}) and the fact that $\Qbf_1$ is $1$-homogeneous:
	\begin{align*}
		\lim_{k \to \infty} [Eu_k]_1(U_R) & = 	\lim_{k \to \infty} \int_{U_R} \Qbf_1\left(\frac{Eu_k}{|Eu_k|}(x)\right) \, d|Eu_k| \mathscr L^n(x)\\
		&  = \int_{U_R} \Qbf_1\left(\frac{E(Tu)}{|E(Tu)|}(x)\right) \, d|E (Tu)|(x) = [E(Tu)](U_R).
	\end{align*} 
For $p > 1$, we recall that $\Qbf_p$ convex, so that $[\frarg]_p$ is lower semicontinuous with respect to weak convergence in $L^p$. This implies the lower bound 
	\[
		[E(Tu)]_p(U) \le \liminf_{k \to \infty} [Eu_k]_p(U).
	\]
	The upper bound follows directly from the strong convergence $Eu_k \to E(Tu)$ in $L^p(\Omega)$ and the triangle inequality for $\Qbf_p$, namely
	\begin{align*}
		\Qbf_p(Eu_k)  \le \Qbf_p(E(Tu)) +  \|Eu_k - E(Tu)\|_{L^p(U)} \to \Qbf_p(E(Tu))\quad \text{in $L^p(U)$}.
	\end{align*}
This proves the claim.
	
\emph{Conclusion.} Let $R \in I$. Using the estimates from Step 1 on $u_k$ we get (recall that $Tu|_U = u$), 
\begin{align*}
	\mathscr F_{p,\eps}(u,U) & = \lim_{k \to \infty} \mathscr F_{\eps,p}(u_k,U) \\
	&\le \lim_{k \to \infty} [Eu_k]_p(U_R)^p + \frac{2}{R^p} \|u_k\|_{L^p(U)}^p\|\rho_\eps\|_{L^1(\R^d \setminus B_R)} \\
	& = M(Tu,U_R) + \frac{2}{R^p} \|u\|_{L^p(U)}^p\|\rho_\eps\|_{L^1(\R^d \setminus B_R)}.
\end{align*}
Letting $\eps \to 0^+$ on both sides of the inequality and recalling~\eqref{eq:eps} yields the estimate
\[
	\limsup_{\eps\to 0^+} \mathscr F_{p,\eps} (u,U) \le M(Tu,U_R).
\]
Now we use that $\Qbf_p(A) \le |A|$, to deduce
\begin{align}\label{eq:cerrado}
\limsup_{\eps\to 0^+} \mathscr F_{p,\eps} (u,U) & \le  M(u,U) + \limsup_{\substack{R \in I,R \to 0^+}} \, M(E(Tu),U_R \setminus U)\nonumber \\
& \le M(u,U) + M(E(Tu),\partial U).
\end{align}
Since $M(E(Tu),\partial \Omega) = 0$, choosing $U = \Omega$ in the estimate above yields the sought upper bound inequality.
\end{proof}

As a immediate corollary we establish pre-compactness (either in $L^p$ or $\Mcal$) for the family $\{\mu_{p,\eps}\}_{\varepsilon>0}$. Moreover, we note  that each of its limit points lies below $|Eu|$. 

\begin{corollary}\label{cor:conv_measures} Let $1 \le p < \infty$ and assume that
	\[
		u \in \begin{cases}
			BD(\Omega) & \text{if $p = 1$},\\
			LD^p(\Omega) & \text{if $1 \le p <\infty$}.
		\end{cases}
	\] 
Then, the family $$\mathscr U \coloneqq \{\mu^p_{p,\eps}\, \mathscr L^d\}_\eps, \qquad \eps \in (0,1)$$
	 is sequentially pre-compact in $\Mcal^+(\Omega)$ with respect to the weak* convergence of measures. Moreover, for every limit 
\[
\mu_{p,\eps_i}^p \toweakstar \mu \quad \text{in $\Mcal^+(\Omega)$}, \qquad \eps_i \to 0^+,
\] 
 there exists a Borel function $g : \Omega \to [0,1]$ satisfying
 \[
 	\mu = g \, \Qbf_p(Eu)^p.
 \]
\end{corollary}

\begin{proof} We give the argument for $p=1$ as the one for $p > 1$ is analogous. The equi-boundedness follows directly from the upper bound and the fact that the measures are positive. Now, let $\mu$ be a limit point as above and let $U \subset \Omega$ be an Borel set. Then, in light of~\eqref{eq:cerrado} we get
	\[
		\mu(U) \le \lim_{\varepsilon_i \to 0^+} \int_\Omega \mu_{p,\eps}^p(U) = \limsup_{\eps \to 0^+} \mathscr F_{p,\eps}(u,U) \le \int_{\bar U \cap \Omega}\Qbf_1(Eu)
	\]
	for all Borel sets $U \subset \Omega$. A standard measure theoretic argument implies that $\mu \ll \Qbf_p(Eu)$. Therefore, by the Radon--Nikod\'ym theorem there exists $g \in L^1(\Omega,\Qbf_1(Eu);\R^+)$ such that 
\[
	\mu = g \,\Qbf_1(Eu), \qquad g \le 1.
\]
This finishes the proof.
\end{proof}

\subsection{Proof of the lower bound} We show that if $p \in (1,\infty)$ and $u \in L^p(\Omega;\R^d)$, then (under the conventions discussed in the introduction)
\[
	[Eu]_p(\Omega)^p \le \liminf_{\eps \to 0^+} \mathscr F_{p,\eps}(u,\Omega) \in [0,\infty],
\]
and  
\[
	[Eu](\Omega) \le \liminf_{\eps \to 0^+} \mathscr F_{1,\eps}(u,\Omega)  \in [0,\infty] \quad \text{for all $u \in L^1(\Omega;\R^d)$.}
\]

For this step, we give a proof by means of a simple mollification argument. This, in turn, differs from the proof by duality originally given in~\cite{BBM} for $W^{1,p}$-gradients. Our proof follows from the observation that the energy is convex with respect to translations, and hence mollification. In particular, our argument  also presents an alternative proof to the one contained in~\cite{D} for gradients, which dispenses with the need of performing certain technical measure theoretic density arguments.

\subsubsection{The regular case}
We begin with a lemma, which establishes a local version of the lower bound for regular functions. 

\begin{lemma}\label{lemma:smooth_lower_bound} 
	Let $\Omega \subset \R^d$ be an open set and let $u \in C^2(\Omega;\R^d)$.
	Then, for every compactly contained connected open set $A \Subset \Omega$ it holds
	\begin{equation*}
		[Eu]_p(A)^p \le \liminf_{\varepsilon \to 0^+} \mathscr F_{p,\varepsilon}(u,A).  
	\end{equation*}
\end{lemma}

\begin{proof}
	Fix $A$ as in the statement and let $x \in A$. Observe that, since $u$ is of class $C^2(\overline{A})$ and $\bar A$ is connected and compact, we may appeal to its Taylor's expansion there. We can thus write, for every $y \in A$
	\begin{equation*}
	u(y) = u(x) + \nabla u(x) \cdot (y-x) + r(\vert x-y \vert ^2),
	\end{equation*}
where $|r(s)| \le Cs$ for every $s \in [0,(\diam A)^2]$, for some constant $C>0$ depending only on $\|u\|_{C^2(\bar A)}$ and $A$.
From this we get 
\begin{equation*}
	\begin{split}
		\frac{\langle u(y) - u(x), y-x \rangle}{|y-x|^2}  & = \frac{\langle \nabla u(x) \cdot (y-x), y-x\rangle }{|y-x|^2} +  \frac{\langle r(|x-y|^2), y-x\rangle}{|y-x|^2} \\ & 
		= \frac{\langle Eu(x) \cdot (y-x), y-x\rangle }{|y-x|^2} +  \frac{\langle r(|x-y|^2), y-x\rangle}{|y-x|^2}. 
	\end{split}
\end{equation*}
Rearranging and taking the absolute values we obtain 
\begin{equation*}
	\begin{split}
		\frac{\vert  \langle E u(x) \cdot (y-x), y-x\rangle \vert}{|x-y|^2}  & \le \frac{\vert \langle u(y) - u(x), y-x \rangle \vert }{|x-y|^2} + \frac{\vert  \langle r(|y-x|^2), y-x\rangle \vert}{|y-x|^2}\\ 
		& \le  \frac{\vert \langle u(y) - u(x), y-x \rangle \vert }{|x-y|^2}  + C| y-x |. \\ 
	\end{split}
\end{equation*}
Recall that, for every $p \ge 1$, for every $a,b \ge 0$ and for every $\zeta>0$, by Young's inequality we have the estimate
\begin{equation*}
	(a+b)^p \le (1+\zeta)a^p + c_{p,\zeta}b^p, 
\end{equation*}
for a suitable (large) constant $c_{p,\zeta}>0$. 
Therefore, we have
\begin{equation*}
	\begin{split}
		&\frac{\vert  \langle E u(x) \cdot(y-x), y-x\rangle \vert^p}{|y-x|^{2p}} \\ & \qquad\le \left(  \frac{\vert \langle u(y) - u(x), y-x \rangle \vert }{|x-y|^2}  + C| x-y | \right)^p \\ 
		& \qquad \le   (1+\zeta)\frac{\vert \langle u(y) - u(x), y-x \rangle \vert^p }{|x-y|^{2p}} + c_{p,\zeta}C| x-y |^p.
	\end{split}
\end{equation*}
Now we multiply times  $\rho_\e(y-x)$ and we integrate w.r.t. $x,y$ on $A$. We obtain 
\begin{equation}\label{eq:collective_double_integral}
	\begin{split}
		&\int_A\int_{A} \frac{\vert  \langle E u(x) \cdot(y-x), y-x\rangle \vert^p}{|y-x|^{2p}}  \rho_\e(y-x)  \, dy\, dx \\ 
		& \qquad \le   (1+\zeta)\int_A \int_{A} \frac{\vert \langle u(y) - u(x), y-x \rangle \vert^p }{|x-y|^{2p}}  \rho_\e(y-x)  \, dy \, dx \\ 
		& \qquad \, + c_{p,\zeta}C \int_A \int_{A}| x-y |^p \rho_\e(y-x) \, dy \, dx. 
	\end{split}
\end{equation}
We observe that the left-hand side of \eqref{eq:collective_double_integral}, for $\delta >0$ sufficiently small, can be estimated from below by
\begin{equation}\label{eq:bound_below_lhs}
	\begin{split}
		&\int_A\int_{A} \frac{\vert  \langle E u(x) \cdot(y-x), y-x\rangle \vert^p}{|y-x|^{2p}}  \rho_\e(y-x)  \, dy\, dx \\ 
		& \ge \int_A\int_{A \cap B_\delta(x)} \frac{\vert  \langle E u(x) \cdot(y-x), y-x\rangle \vert^p}{|y-x|^{2p}}  \rho_\e(y-x)  \, dy \, dx  \\ 
		& \ge  \int_{A_{-\delta} } \int_{B_\delta(x)}\frac{\vert  \langle E u(x) \cdot(y-x), y-x\rangle \vert^p}{|y-x|^{2p}}  \rho_\e(y-x)  \, dy \, dx
	\end{split}
\end{equation}
where $A_{-\delta}:=\{x \in A: B_\delta(x) \subset A\}$. 
Combining \eqref{eq:collective_double_integral} and \eqref{eq:bound_below_lhs} we have 
\begin{equation}\label{eq:boh}
	\begin{split}
		&\int_{A_{-\delta} } \int_{B_\delta(x)}\frac{\vert  \langle E u(x) \cdot(y-x), y-x\rangle \vert^p}{|y-x|^{2p}}  \rho_\e(y-x)  \, dy \, dx  \\ 
		& \qquad \le   (1+\zeta)\int_A \int_{A} \frac{\vert \langle u(y) - u(x), y-x \rangle \vert^p }{|x-y|^{2p}}  \rho_\e(y-x)  \, dy \, dx \\ 
		& \qquad \, + c_{p,\zeta}C \int_A \int_{A}| x-y |^p \rho_\e(y-x) \, dy\,dx. 
	\end{split}
\end{equation}
We now send $\e \to 0^+$ in \eqref{eq:boh} and we study the terms separately.  

\emph{Claim 1.} We have
\begin{equation*}
	\lim_{\varepsilon \to 0^+}  \int_{ A_{-\delta}} \int_{B_\delta(x)} \frac{\vert  \langle E u(x) \cdot(y-x), y-x\rangle \vert^p}{|y-x|^{2p}}  \rho_\e(y-x)  \, dy\, dx = [Eu]_p(A_{-\delta})^p. 
\end{equation*}
Indeed, by the change of variables $h:=y-x$ (in the $y$-variable) we can re-write 
\begin{equation*}
	\begin{split}
		 \int_{ A_{-\delta}}\int_{B_\delta(x)} &  \frac{\vert \langle E u(x) \cdot (y-x), y-x\rangle \vert^p }{|y-x|^{2p}} \rho_\e(y-x)\, dy \, dx \\ 
		= &  \int_{A_{-\delta}} \int_{B_\delta(0)}  \frac{\vert \langle E u(x) \cdot h, h\rangle \vert^p }{|h|^{2p}} \rho_\e(h)\, dh \, dx. 
	\end{split}
\end{equation*}
Therefore, by the coarea formula on balls and the radial symmetry of the mollifier, we further obtain 
\begin{equation*}
	\begin{split}
		\int_{A_{-\delta}} \int_{B_\delta}  & \frac{\vert \langle E u(x) \cdot h, h\rangle \vert^p }{|h|^{2p}} \rho_\e(h)\, dh \, dx \\  
		& = \int_{A_{-\delta}} \fint_{\mathbb S^{d-1}} \vert \langle E u(x) \cdot \omega, \omega \rangle \vert^p\,  dS(\omega) \, dx  \int_0^\delta |\partial B_1|\widehat{\rho_\varepsilon}(r)r^{d-1} \, dr \\ 
		& = \Vert \rho_\varepsilon\Vert_{L^1(B_\delta)} \int_{A_{-\delta}} \mathbf Q_p(Eu(x))^p \, dx \, \stackrel{\eqref{eq:eps1}}\longrightarrow\, [Eu]_p(A_{-\delta})^p
	\end{split}
\end{equation*} 
as $\varepsilon \to 0^+$ and this concludes the proof of the Claim. 

\emph{Claim 2.} We have  
\begin{equation*}
	\lim_{\varepsilon \to 0^+} \int_{A} \int_{A}  |x-y |^p  \rho_\e(y-x) \, dy \, dx = 0. 
\end{equation*}
Changing variables as in the previous claim, we get 
\begin{equation*}
	\int_{A} \int_{A}  |x-y |^p  \rho_\e(y-x) \, dy \, dx \le \mathscr L^d(A) \int_{B_R} |h |^p  \rho_\e(h) \, dh,
\end{equation*}
where $R  \coloneqq \diam A$.
Now, for any fixed $\sigma >0$, we can estimate the integral on the right-hand side by
\begin{equation*}
	\begin{split}
		\int_{\R^d} |h|^p \rho_\varepsilon(h) \, dh & \le \int_{B_{\sigma}} |h|^p \rho_\varepsilon(h) \, dh +   R^p\int_{B_R\setminus  B_{\sigma}}  \rho_\varepsilon(h) \, dh. 
	\end{split}
\end{equation*}
The second term on the right-hand side vanishes as $\e \to 0^+$ due to \eqref{eq:eps}. Recalling the normalization condition \eqref{eq:eps1}, the first term on the right-hand side can be roughly estimated by
\begin{equation*}
	\limsup_{\varepsilon \to 0^+} \int_{B_\sigma} |h|^p \rho_\varepsilon(h) \, dh \le \sigma^p
\end{equation*}
and, since $\sigma>0$ is an arbitrary positive real number, we conclude the assertion of Claim 2. 

Combining Claim 1 and Claim 2, we have from \eqref{eq:boh} 
\begin{equation}\label{eq:before_limit_delta}
	[Eu]_p(A_{-\delta})^p \le (1+\zeta) \liminf_{\eps \to 0^+} \mathscr F_{\e, p}(u,A)
\end{equation} 
and, since $\zeta>0$ is arbitrary, 
\begin{equation*}
	[Eu]_p(A_{-\delta})^p \le \liminf_{\eps \to 0^+} \mathscr F_{\e, p}(u,A). 
\end{equation*} 
We now observe that, since $A$ is an open set, $A_{-\delta} \uparrow A$ as $\delta \to 0^+$. Since positive measures are continuous along monotone sequences \cite[Remark 1.3]{AFP}, we can pass to the limit in \eqref{eq:before_limit_delta}, obtaining 
\begin{equation*}
	[Eu]_p(A)^p \le \liminf_{\eps \to 0^+} \mathscr F_{\e, p}(u,A),
\end{equation*} 
which is the sought estimate. \end{proof}

\subsubsection{The general case} 
We are now ready to discuss the lower bound in the general case of a function $u \in L^p(\Omega;\R^d)$. We will make use of the following observation (which seems to be originally due to E. Stein as mentioned in~\cite{brezis2,ponce}). We will denote by $(\psi_{\eta})_{\eta} \subset C^{\infty}_c(\R^d)$  a family of non-negative smoothing kernels, with $\supp\psi_{\eta} \subset B_{\eta}(0)$ and $\int_{\R^d} \psi_{\eta} = 1$ for every $\eta >0$. For every $u \in L^1(\Omega;\R^d)$ we define $u_{\eta}(x):= (u\ast \psi_{\eta})(x)$, which is well-defined and smooth in the set $\Omega_\eta:= \{x \in \Omega: d(x,\partial \Omega) > \eta\}$. 

\begin{lemma}\label{lemma:convolutions}
	Let $\Omega \subset \R^d$ be an open set and $u \in L^p(\Omega;\R^d)$. Let $A$ be an open set with $A \Subset \Omega$. 
	Then for every $p \in [1,\infty)$, for every $\e>0$ and for every $0 < \eta < \frac{1}{2} \dist(A,\partial\Omega)$ it holds 
	\begin{equation*}
		\mathscr F_{p,\varepsilon} (u_\eta,A) \le \mathscr F_{p,\varepsilon} (u,A_{\eta}), 
	\end{equation*}
	where $A_\eta := A + B_\eta(0)$ is the open neighborhood of $A$ of radius $\eta$.
\end{lemma}

\begin{proof}
	The claim is a rather easy consequence of the convexity of $\mathscr F_{p,\varepsilon}$ which in turn follows by Jensen's inequality. Indeed, we have 
	\begin{equation*}
		\begin{split}
			& \iint_{A \times A}   \vert \langle u_{\eta}(x) - u_{\eta}(y),x-y \rangle |^p \frac{\rho_\varepsilon(x -y)}{{|x - y|^{2p}} } \, dxdy \\
			= & \iint_{A\times A}  \left \vert \left \langle \int_{\R^d}[ u(x-z) - u(y-z) ]\psi_{\eta}(z)\, dz, x-y \right \rangle \right\vert^p  \frac{\rho_\varepsilon(x -y)}{|x - y|^{2p}} \, dxdy \\
			= &\iint_{A\times A} \left \vert  \int_{\R^d} \left \langle  u(x-z) - u(y-z), x-y \right  \rangle  \psi_{\eta}(z)\, dz \right\vert^p \frac{\rho_\varepsilon(x -y)}{|x - y|^{2p}} \, dxdy \\
			\le  &\iint_{A \times A}  \int_{\R^d}\left \vert  \left \langle  u(x-z) - u(y-z), x-y \right \rangle \right\vert^p \psi_{\eta}(z)\, dz \,  \frac{\rho_\varepsilon(x -y)}{|x - y|^{2p}} \, dxdy \\
		\end{split}
	\end{equation*}
	where the last inequality indeed follows by Jensen's inequality with respect to the probability measure $\psi_{\eta}\L^d$. An application of Tonelli's theorem and a change of variables yield 
	\begin{equation*}
		\begin{split}
			\int_{\R^d}\ &\iint_{A \times A} \left \vert  \left \langle  u(x-z) - u(y-z), x-y \right \rangle \right\vert^p   \frac{\rho_\varepsilon(x -y)}{|x - y|^{2p}} \, dxdy \, \psi_{\eta}(z)\, dz \\
			\le  & \int_{\R^d} \iint_{A_{\eta} \times A_{\eta}} \left \vert  \left \langle  u(w) - u(s), w-s \right \rangle \right\vert^p   \frac{\rho_\varepsilon(w -s)}{|w -s|^{2p}} \, dwds \, \psi_{\eta}(z)\, dz  \\
		\end{split}
	\end{equation*}
	and, recalling that $\Vert \psi_\eta\Vert_{L^1} =1$, this concludes the proof. 
\end{proof}

We are now ready to present the proof of the lower bound. 
\begin{proof}[Proof of the lower bound] Let $u \in L^p(\Omega;\R^d)$. It is not restrictive to assume that 
	\begin{equation*}
		L\coloneqq\liminf_{\eps \to 0^+} \mathscr F_{p,\varepsilon} (u, \Omega)
	\end{equation*}
	is finite (otherwise there is nothing to prove). Let $A \Subset \Omega$ be an open set whose closure is compact and contained in $\Omega$. 
	Consider convolution kernels $(\psi_\eta)_\eta$ as above. For sufficiently small $\eta>0$, the associated smooth approximations $u_\eta$ are well defined and smooth in $A \subset \Omega_\eta \subset \Omega$. We can therefore use Lemma \ref{lemma:smooth_lower_bound} and obtain
	\begin{equation}\label{eq:loc_u_eta} 
		[Eu_\eta]_p(A)^{p} \le\liminf_{\eps \to 0^+} \mathscr F_{p,\varepsilon} (u_\eta,A). 
	\end{equation}
	In turn, by Lemma \ref{lemma:convolutions}, we get 
	\begin{equation}\label{eq:convex_convolution0}
		\mathscr F_{p,\varepsilon} (u_\eta, A) \le \mathscr F_{p,\varepsilon} (u, A_\eta) \le  \mathscr F_{p,\varepsilon} (u, \Omega ),
	\end{equation}
	where we have used the fact that $A_\eta = A+B_\eta \subset \Omega$ for sufficiently small $\eta>0$. Passing to the limit in $\e \to 0^+$ in \eqref{eq:convex_convolution0} and combining it with \eqref{eq:loc_u_eta}, we obtain 
	\begin{equation*}
		C_{d,p} |Eu_\eta|(A)^p \le [Eu_\eta]_p(A)^{p} \le L. 
	\end{equation*}
	In particular, the family $\{u_\eta\}_\eta$ has $L^p$-bounded symmetric gradients on $A$ and hence by the lower semicontinuity of the classical total variation we conclude that $Eu \in \Mcal_b(A;M^{d \times d}_\sym)$. Now, we may use the lower semicontinuity of the $\Qbf_p$-variation (cf. the argument given in p.~10) to deduce that 
	\begin{equation}\label{eq:est1}
		\begin{split}
			[Eu]_p(A)^{p} & \le L \; \text{if $p > 1$},\\
			[Eu](A) & \le L \; \text{if $p = 1$}.
		\end{split}
	\end{equation}
	Since $\Omega$ connected and open, it can be written as the monotone limit $A_j \toup \Omega$ of connected open sets, whose closure is compact and contained in $\Omega$. Exploiting again the continuity of positive measures along monotone sequences \cite[Remark 1.3]{AFP}, we may then pass to the limit as $A_j \uparrow \Omega$ in \eqref{eq:est1} and obtain the desired lower bound 		
	\begin{equation}\label{eq:final_lower_bound}
				\begin{split}
					[Eu]_p(\Omega)^{p} & \le L \; \text{if $p > 1$},\\
					[Eu](\Omega) & \le L \; \text{if $p = 1$}.
				\end{split}
			\end{equation}
	This implies that $u \in BD(\Omega)$ if $p = 1$ or $u \in LD^p(\Omega)$ if $p > 1$ and concludes the proof. 
\end{proof}

	\subsection{Proof of Theorems~\ref{thm:m} and~\ref{thm:1}} Let $p \in [1,\infty)$. The existence and characterization of the extended limits
\[
\lim_{\eps \to 0^+} \mathscr F_{p,\eps}(u,\Omega) \in [0,\infty],
\]
for arbitrary functions $u \in L^p(\Omega;\R^d)$,
follows directly from the lower and upper bounds.

\subsection{Proof of Corollaries~\ref{cor:p} and~\ref{cor:BD}} The proof of these two convergence results follows directly from Corollary~\ref{cor:conv_measures} and Theorems~\ref{thm:m},~\ref{thm:1}. Indeed, if $\mu = g\Qbf_p(Eu)$ is the measure from Corollary~\ref{cor:conv_measures}, then the characterizations imply that (in both cases) $g \equiv 1$. Since Corollary~\ref{cor:conv_measures} is valid for arbitrary subsequences of $(\mu_{p,\eps})$, this shows that
\[
\mu_{p,\eps} \toweak \Qbf_p(Eu) \quad \text{in $L^p(\Omega)$ for $p \in (1,\infty)$,}
\]
and 
\[
\mu_{1,\eps}\,\mathscr L^d  \toweakstar [Eu] \quad \text{in $\Mcal(\Omega)$.}
\]
Moreover, from Theorems~\ref{thm:m} and~\ref{thm:1} it follows that
\[
\|\mu_{p,\eps}\|_{L^p(\Omega)} \to \|\Qbf_p\|_{L^p(\Omega)} \quad \text{for $p > 1$},
\]
and
\[
|\mu_{1,\eps} \mathscr L^d|(\Omega) \to [Eu](\Omega),
\]
for the extended values of the norm. Hence, the equivalences in Corollary~\ref{cor:p} follow directly from the convergence of the $L^p$-norms, while  Corollary~\ref{cor:BD} follows verbatim from the aforementioned convergences when $p = 1$.


\section{A functional for the singular part}
We now show that subtracting the first-order term of the approximate Taylor polynomial of $u$ to our symmetric difference quotient leads to a limiting functional representation of the singular part $E^su$. By appealing to similar ideas as the ones introduced in the previous section via the extension operator $T \colon BD(\Omega) \to BD(\R^d)$, it suffices to prove the following proposition:
\begin{proposition}
	Let $u \in BD(\R^d;\R^d)$ and let $\nabla u(x) \in M^{d \times d}$ denote the approximate differential of $u$ at  a point $x$, which exists $\mathscr L^d$-almost everywhere in $\R^d$. Then, the extended limit
	\begin{equation*}
	\lim_{\varepsilon \to 0^+} \iint_{\R^d \times \R^d}  \frac{\vert \langle u(x) - u(y) -\mathcal Eu(x) [y-x],y-x\rangle|}{|x - y|^2} \, \rho_\varepsilon(x -y) \, dx\,dy
	\end{equation*}
exists and equals $[E^s u](\R^d)$.
\end{proposition}

\begin{proof} We follow closely the ideas contained in~\cite[Sect.~2]{ponce}. From the Radon--Nikod\'ym--Lebesgue decomposition of $Eu$, we find that (with the same notation of the previous section)
\begin{equation}\label{eq:deceiving}
	E(u_\eta )= \mathcal (\mathcal Eu)_\eta + (E^s u)_\eta.
\end{equation}
Define the energies 
\begin{equation*}
	\mathcal R^{\varepsilon}u(x) :=  \int_{\R^d} \frac{\vert \langle u(y) - u(x) - \mathcal Eu(x) [y-x],y-x\rangle|}{|x - y|^2} \, \rho_\varepsilon(y-x) \, dy.
\end{equation*}
Appealing to similar Jensen inequalities as the ones in the previous section, we  find    
\[
\int_{\R^d} \frac{\vert \langle u_\eta(x+h) - u_\eta(x) - (\mathcal Eu)_\eta(x) [h],h\rangle|}{|h|^2} \, \rho_\varepsilon(h) \, d{{h}}  \le (\mathcal{R}^\eps u)_\eta(x)
\]
and hence Young's inequality gives
\begin{equation}\label{eq:R1}
\iint_{\R^d \times \R^d}  \frac{\vert \langle u_\eta(y) - u_\eta(x) - (\mathcal Eu)_\eta(x) [y-x],y-x\rangle|}{|x - y|^2} \, \rho_\varepsilon(y-x) \, dy \le \|\mathcal R^\eps u\|_{L^1}.
\end{equation}
In light of~\eqref{eq:deceiving}, the triangle inequality and the identity
\begin{align*}
	\Qbf_1(A) & = \int_0^\infty \fint_{\partial B_1} {|\dpr{A\omega,\omega}|}\,dS(\omega) \, \hat\rho_\eps(r) |\partial B_1| r^{d-1} \, dr \\
	& = \int_{\R^d}  \frac{|\dpr{Ah,h}|}{|h|^2} \, \rho_\eps(h) \, dh, \qquad \eps > 0,
\end{align*}
we deduce that $\Rcal^\eps u_\eta(x)$ is a bound for the energy $(\Jrm)_\eps(x)$, defined by
\[
\left|\int_{\R^d}  \frac{\vert \langle u_\eta(x+h) - u_\eta(x) - (\mathcal Eu)_\eta(x) [h],h\rangle|}{|h|^2} \, \rho_\eps(h) \, dh - \Qbf_1((E^su)_\eta(x))\right|.
\]
Applying once more the triangle inequality and integrating over $x$, we deduce from the previous bound,~\eqref{eq:R1} and Young's convolution inequality that
\begin{align*}
	[(E^s u)_\eta](\R^d) & = \int_{\R^d} \Qbf_1((E^s u)_\eta(x)) \, dx\\
	 & \le \int_{\R^d}(\Jrm)_\eps(x) \, dx + \int_{\R^d} (\Rcal^\eps u)_\eta(x) \, dx \\
	& \le \|\Rcal^\eps u_\eta\|_{L^1} + \|\Rcal^\eps u\|_{L^1}.
\end{align*}
Now, since $u_\eta \in (C^\infty \cap W^{1,1})(\Omega)$, then the first term on the right-hand side vanishes as $\eps \to 0^+$. On the other hand, by Reshetnyak's continuity theorem and the convergence $(E^s u)_\eta \toweakstar E^s u$, we find that
\[
	[E^s u](\R^d) \le \liminf_{\eps \to 0^+} \|\Rcal^\eps u\|_{L^1}.
\]
This proves the lower bound.

For the upper bound, we continue following \cite{ponce} and we claim that for every non-negative bounded continuous function $\varphi \colon \R^d \to \R$ it holds
\begin{equation}\label{eq:final}
	\int_{\R^d} \mathcal R^{\varepsilon}u(x) \varphi(x) \, dx \le \int_{\R^d} \varphi d[E^s u] + {\rm (II)}_{\e} + {\rm(III)}_{\e}  
\end{equation}
where 
\begin{equation*}
	{\rm(II)}_{\e} := \int_0^1 \int_{\R^d} \left( \int_{\R^d} |\varphi(x+th) -\varphi(x)| \, \rho_{\e}(h) \, dh \right) \, d|E^s u|(x) \,dt
\end{equation*}
and 
\begin{equation*}
	{\rm(III)}_{\e} := \Vert \varphi \Vert_{\infty} \int_0^1 \int_{\R^d} \left( \int_{\R^d} |\mathcal E u(x+th) -\mathcal E u(x)| \, dx \right) \rho_{\e}(h) \, dh \,dt.
\end{equation*}
The conclusion will then follow observing that ${\rm(II)}_\e \to 0^+$ (because $\varphi$ is bounded and continuous) and ${\rm(III)}_{\e} \to 0$ as $\e \to 0$ (because $\mathcal E u$ is an $L^1$-function). To establish \eqref{eq:final}, we rely once again on an approximation argument: by the Fundamental Theorem of Calculus and \eqref{eq:deceiving} we infer 
\begin{equation}\label{eq:fundamental_thm}
	\begin{split}
	| \langle u_\eta(x+h) -u_\eta(x) -  &(\mathcal E u)_\eta(x)\cdot h, h \rangle  | \\
	 & \quad \le  \int_0^1 |\langle (E^s u)_\eta(x+th) \cdot h,h \rangle | \, dt \\ 
	& \qquad +   \int_0^1 | (\mathcal E u)_\eta(x+th) - (\mathcal E u)_\eta(x) | |h|^2 \, dt. 
\end{split}
\end{equation}
Now, exactly as in \cite{ponce}, we observe that, adding and subtracting the term $\varphi(x+th)$ and changing variables $z:=x+th$, one has 
\begin{equation*}
	\begin{split}
		 \int_{\R^d}  & | \langle (E^s u)_\eta(x+th)\cdot h, h \rangle | \varphi(x) \,dx \\  
		& \le \int_{\R^d}   | \langle (E^s u)_\eta(z)h,h \rangle | \varphi(z) \,dz \\
		& \quad + |h|^2 \int_{\R^d} | (E^s u)_\eta(z)| |\varphi(z) - \varphi(z-th)|\,dz. 
	\end{split}
\end{equation*}
In conclusion, by Fubini Theorem 
\begin{equation}\label{eq:almost} 
	\begin{split}
		 \int_{\R^d} \int_{\R^d} & \int_0^1 \frac{ \vert \langle (E^s u)_\eta(x+th) \cdot h, h \rangle \vert}{|h|^2} dt \, \rho_\varepsilon(h) \, dh\, \varphi(x)dx \\ 
		& =  \int_{\R^d} \int_{\R^d} \frac{| \langle (E^s u)_\eta(z)h,h \rangle |}{|h|^2} \rho_\varepsilon(h) \, dh \, \varphi(z) \,dz  + {\rm (II)}_{\e,\eta} \\  
		& =  \int_{\R^d} \varphi(z) \Qbf_1((E^su)_\eta(z)) \,dz  + {\rm (II)}_{\e,\eta} 
	\end{split}
\end{equation}
where we have denoted by 
\begin{equation*}
	{\rm (II)}_{\e,\eta} := \int_0^1 \int_{\R^d} \int_{\R^d} | \varphi(z) - \varphi(z-th)| \, | (E^s u)_\eta(z)  |\,dz\,\rho_\varepsilon(h) \, dh\,  dt .
\end{equation*}
Observe that due to Jensen's inequality we have the point-wise inequalities $|(E^s u)_\eta| \le |E^s u| \ast \psi_\eta$ and $[(E^s u)_\eta] \le [E^s u] \ast \psi_\eta$. In particular, ${\rm (II)}_{\e,\eta} \le {\rm (II)}_{\e}$ which, combined with \eqref{eq:almost} and \eqref{eq:fundamental_thm}, yields \eqref{eq:final} and the proof is complete. 
\end{proof}

\section{Acknowledgements} 
	This project has received funding from the European Research Council (ERC) under the European Union’s Horizon 2020 research and innovation programme, grant agreement No. 757254 (SINGULARITY).

\end{document}